\documentclass{amsart}

\newtheorem{theorem}{Theorem}[section]
\newtheorem{lemma}[theorem]{Lemma}

\theoremstyle{definition}
\newtheorem{definition}[theorem]{Definition}
\newtheorem{example}[theorem]{Example}

\newtheorem{corollary}[theorem]{Corollary}
\theoremstyle{remark}
\newtheorem{remark}[theorem]{Remark}

\numberwithin{equation}{section}

\begin{document}

\title[Clairaut Anti-Invariant Submersions]{Clairaut Anti-Invariant Submersions from \\normal almost contact metric manifolds}


\author[Ta\c{s}tan, H.M.]{Hakan Mete Ta\c{s}tan}
\address{\.{I}stanbul University, Faculty of Science, Department of Mathematics, Vezneciler, 34134, \.{I}stanbul, Turkey}
\curraddr{}
\email{hakmete@istanbul.edu.tr}
\thanks{}

\author[Gerdan, S.]{S\.{I}bel Gerdan}
\curraddr{}
\email{sibel.gerdan@istanbul.edu.tr }
\thanks{}



\date{}

\dedicatory{}

\begin{abstract}
We investigate new Clairaut conditions for anti-invariant submersions from normal almost contact metric manifolds onto Riemannian manifolds. We prove that there is no Clairaut anti-invariant submersion admitting vertical Reeb vector field when the total manifold is Sasakian. Several illustrative examples are also included.
\end{abstract}

\maketitle

\section{Introduction}
In the theory of surfaces, Clairaut's theorem states that for any geodesic $\alpha$ on a surface $S$, the function $r\sin \theta$ is constant along $\alpha$, where $r$ is the distance from a point on the surface to the rotation axis and $\theta$ is the angle between $\alpha$ amd the meridian through $\alpha$. This idea was applied to the Riemannian submersions \cite{O} by Bishop \cite{Bi} and he gave a necessary and sufficient condition for a Riemannian submersion to be Clairaut. Allison \cite{Al} considered Clairaut submersions when the total manifolds were Lorentzian and he also showed that such submersions have interesting applications in static space-times. Lee et al. \cite{Le3}, investigated new conditions for anti-invariant Riemannian submersions \cite{Sa2} to be Clairaut when the total manifolds are K$\ddot{a}$hlerian. A similar study \cite{Sa5} was done by \c{S}ahin and the first author of this paper for semi-invariant submersions \cite{Sa3}, slant submersions \cite{Sa4} and pointwise slant submersions \cite{Le2}. \\

In the present paper, we consider anti-invariant Riemannian submersions from normal almost contact metric manifolds onto Riemannian manifolds.
After giving a necessary and sufficient condition for a curve on the total manifolds to be geodesic, we focus investigate new Clairaut conditions for
considered submersions. We first give a new necessary and sufficient condition for anti-invariant submersions admitting horizontal Reeb vector field to be Clairaut in the case of the total manifolds are Sasakian. We also give a characterization for such submersions when they satisfy Clairaut condition. Contrary to the case of admitting horizontal Reeb vector field, we prove that there is no anti-invariant submersion satisfying Clairaut condition in the case of admitting vertical Reeb vector field when the total manifold is Sasakian. Finally, we present a new necessary and sufficient condition for anti-invariant submersions to be Clairaut in the case of their total manifolds are Kenmotsu. An illustrative example for each kind of submersion is also given.\\

\section{Preliminiaries}
This section consists of four subsections. In subsection 2.1, we present the fundamental definitions and notions
some classes of normal almost contact metric manifolds such as Sasakian and Kenmotsu. In subsection 2.2,
we give the basic background for Riemannian submersions. In subsection 2.3, we recall the fundamental definitions and notions of anti-invariant Riemannian and Lagrangian submersions. The definition and a characterization of Clairaut submersions are placed in the last subsection.
\subsection{Some classes of normal almost contact metric manifolds}
Let $(M,g)$ be a $(2m+1)$-dimensional Riemannian manifold and denote by $TM$ the set of vector fields of $M.$
Then, $M$ is called an \emph{almost contact metric manifold} \cite{Bla} if there exists a tensor $\varphi$ of type $(1,1)$
and global vector field $\xi$ which is called the \emph{Reeb vector field} or the \emph{characteristic vector field}
such that for any $E,F\in TM$, we have
\begin{eqnarray} \nonumber
\varphi\xi=0\,,\quad\eta(\xi)=1\,,\quad\varphi^{2}=-I+\eta\otimes\xi
\end{eqnarray}
and
\begin{eqnarray}  \label{new1}
g(\varphi E, \varphi F)=g(E,F)-\eta(E)\eta\,,
\end{eqnarray}
where $\eta$ is the dual 1-form of $\xi$. Also, it can be deduced from the above axioms that
\begin{eqnarray}  \nonumber
\eta\circ\varphi=0\quad and \quad \eta(E)=g(E,\xi)\,.
\end{eqnarray}
In this case, $(\varphi,\xi,\eta,g)$ is called the \emph{almost contact metric structure} of $M.$
The almost contact metric manifold $(M,\varphi,\xi,\eta,g)$ is called a \emph{contact metric manifold} if we have
$$\Phi(E,F)=d\eta(E,F)$$
for any $E,F\in TM$, where $\Phi$ is a 2-form in $M$ defined by $g(E,\varphi F)$. The 2-form $\Phi$ is called the \emph{fundamental 2-form} of $M$.
A contact metric structure $(\varphi,\xi,\eta,g)$ of $M$ is said to be \emph{normal} \cite{Yan} if we have
$$[\varphi,\varphi]+2d\eta\otimes\xi=0\,,$$
where $[\varphi,\varphi]$ is Nijenhuis tensor of $\varphi$.
Any normal contact metric manifold is called a \emph{Sasakian manifold.}
It is not difficult to prove that a contact metric manifold $M$ is a Sasakian manifold if and only if
\begin{eqnarray}\label{e12}
(\nabla_{E}\varphi)F=g(E,F)\xi-\eta(F)E
\end{eqnarray}
for any $E,F\in TM$, where $\nabla$ denotes the Levi-Civita connection of $M$.
For the further information of Sasakian manifolds, see the classical books \cite{Bla,Yan}.\\

A \emph{Kenmotsu manifold} $M$ \cite{Ke} is a normal  almost contact metric manifold satisfying
\begin{eqnarray}\label{f1}
(\nabla_{E}\varphi)F=g(\varphi E,F)\xi-\eta(F)\varphi E.
\end{eqnarray}
for all $E,F\in TM$. We refer to the original paper \cite{Ke} for fundamental definitions and notions of Kenmotsu manifolds.\\

\subsection{Riemannian submersions}
Let $(M,g)$ and $(N,g_{\text{\tiny$N$}})$ be Riemannian manifolds,
where $dim(M)>dim(N)$. A surjective mapping
$\pi:(M,g)\rightarrow(N,g_{N})$ is called a \emph{Riemannian
submersion}
\cite{O} if:\\

\textbf{(S1)} The rank of $\pi$ is equal to $dim(N).$  \\

In which case, for each $q\in N$, $\pi^{-1}(q)=\pi_q^{-1}$ is a $k$-dimensional
submanifold of $M$ and called a \emph{fiber}, where $k=dim(M)-dim(N).$
A vector field on $M$ is called \emph{vertical} (resp.
\emph{horizontal}) if it is always tangent (resp. orthogonal) to
fibers. A vector field $Y$ on $M$ is called \emph{basic} if $Y$ is
horizontal and $\pi$-related to a vector field $Y_{*}$ on $N,$ i.e.,
$\pi_{*}Y_{p}=Y_{*\pi(p)}$ for all $p\in M,$ where $\pi_{*}$ is the derivative map of $\pi$. As usual, we denote by
$\mathcal{V}$ and $\mathcal{H}$ the projections on the vertical
distribution $ker\pi_{*}$ and the horizontal distribution
$(ker\pi_{*})^{\bot},$ respectively.\\

\textbf{(S2)} For all $p\in M$ and for any horizontal vectors $Y$ and $Z$ at $p$ and, we have
\begin{equation}\nonumber
g(Y_{p},Z_{p})=g_{\text{\tiny$N$}}(\pi_{*}Y_{p},\pi_{*}Z_{p})\,,
\end{equation}
that is, $\pi_{*}$ preserves lengths of horizontal vectors.\\

The geometry of Riemannian submersions is characterized by O'Neill's tensors $\mathcal{T}$ and
$\mathcal{A}$, defined as follows:
\begin{equation}\label{e1}
\mathcal{T}_{E}F=\mathcal{V}\nabla_{\mathcal{V}E}\mathcal{H}F+\mathcal{H}\nabla_{\mathcal{V}E}\mathcal{V}F,
\end{equation}
\begin{equation}\label{e2}
\mathcal{A}_{E}F=\mathcal{V}\nabla_{\mathcal{H}E}\mathcal{H}F+\mathcal{H}\nabla_{\mathcal{H}E}\mathcal{V}F
\end{equation}
for any vector fields $E$ and $F$ on $M,$ where $\nabla$ is the
Levi-Civita connection of $g$. It is easy to see
that $\mathcal{T}_{E}$ and $\mathcal{A}_{E}$ are skew-symmetric
operators on the tangent bundle of $M$ reversing the vertical and
the horizontal distributions. We summarize the properties of the
tensor fields $\mathcal{T}$ and $\mathcal{A}$. Let $W,U$ be
vertical and $Y,Z$ be horizontal vector fields on $M$, then we
have
\begin{equation}\label{e3}
\mathcal{T}_{W}U=\mathcal{T}_{U}W,
\end{equation}
\begin{equation}\label{e4}
\mathcal{A}_{Y}Z=-\mathcal{A}_{Z}Y=\frac{1}{2}\mathcal{V}[Y,Z].
\end{equation}
Equation \eqref{e3} says that $\mathcal{T}$ is symmetric for vertical vector fields, while equation \eqref{e4} says that $\mathcal{A}$ is skew symmetric for horizontal vector fields. Moreover, from \eqref{e4} it follows the horizontal distribution is integrable if and only if $\mathcal{A}$ is zero, identically.  On the other hand, from (1) and (2), we obtain
\begin{equation}\label{e5}
\nabla_{W}U=\mathcal{T}_{W}U+\hat{\nabla}_{W}U,
\end{equation}
\begin{equation}\label{e6}
\nabla_{W}Y=\mathcal{T}_{W}Y+\mathcal{H}\nabla_{W}Y,
\end{equation}
\begin{equation}\label{e7}
\nabla_{Y}W=\mathcal{A}_{Y}W+\mathcal{V}\nabla_{Y}W,
\end{equation}
\begin{equation}\label{e8}
\nabla_{Y}Z=\mathcal{H}\nabla_{Y}Z+\mathcal{A}_{Y}Z,
\end{equation}
where $\hat{\nabla}_{W}U=\mathcal{V}\nabla_{W}U$.
Moreover, if $Y$ is basic, then we have
\begin{eqnarray}
\mathcal{H}\nabla_{W}Y=\mathcal{A}_{Y}W.
\end{eqnarray}
From \eqref{e5},  we see that $\mathcal{T}$ acts on the fibers as the second fundemantal form. We also observe that the horizontal distribution is totally geodesic if and only if $\mathcal{A}$ is zero, identically from \eqref{e8}.  For  details on the Riemannian
submersions, we refer to the papers, \cite{Gr,O} and to the books \cite{Fa,Sa1}.
\subsection{Anti-invariant Riemannian  submersions}
The notion of anti-invariant Riemannian submersion was first defined by \c{S}ahin \cite{Sa1} in almost Hermitian geometry and then
this notion was applied to almost contact geometry by Lee \cite{Le1} as follows.
\begin{definition} (\cite{Le1}) Let $M$ be a $(2m+1)$-dimensional almost contact metric manifold with
almost contact metric structure $(\varphi,\xi,\eta,g)$
and $N$ be a Riemannian manifold with Riemannian metric
$g_{\text{\tiny$N$}}.$ Suppose that there exists a Riemannian
submersion $\pi:M\rightarrow N$ such that the vertical distribution $ker\pi_{*}$ is
anti-invariant with respect to $\varphi,$ i.e., $\varphi ker\pi_{*}\subseteq
ker\pi_{*}^{\bot}.$ Then the Riemannian submersion $\pi$ is called
an \emph{anti-invariant Riemannian submersion}. We will briefly call such submersions as anti-invariant submersions.
\end{definition}
In this case, the horizontal distribution $ker\pi_{*}^{\bot}$ is
decomposed as
\begin{equation} \label{e15}
ker\pi_{*}^{\bot}=\varphi ker\pi_{*}\oplus\mu\,,
\end{equation}
where $\mu$ is the orthogonal complementary distribution of
$\varphi ker\pi_{*}$ in $ker\pi_{*}^{\bot}$ and it is invariant
with respect to $\varphi.$\\

We say that an anti-invariant Riemannian submersion $\pi:M\rightarrow N$ \emph{admits vertical Reeb vector field}
if the Reeb vector field $\xi$ is tangent to $ker\pi_{*}$ and it \emph{admits horizontal vector Reeb vector field}
if the Reeb vector field $\xi$ is normal to $ker\pi_{*}.$ It is easy to see that $\mu$ contains the Reeb vector field $\xi$
in the case of $\pi:M\rightarrow N$ admits horizontal vector Reeb vector field $\xi$. For any $Y\in ker\pi_{*}^{\bot}$, we write
\begin{eqnarray}
\label{f0}
\varphi Y=\mathcal{B}Y+\mathcal{C}Y\,,
\end{eqnarray}
where  $\mathcal{B}Y \in ker\pi_{*}$ and $\mathcal{C}Y\in ker\pi_{*}^{\bot}$.\\

For some details and examples of the anti-invariant Riemannian submersions from almost contact metric manifold $(M,\varphi,\xi,\eta,g)$ onto a Riemannian manifolds $(N,g_N)$, we refer to the papers \cite{Be,E,Le1} and to the book \cite{Sa1}.
\begin{definition} (\cite{Ta2}) Let $\pi$ be an anti-invariant Riemannian submersion from  an almost contact metric manifold $(M,\varphi,\xi,\eta,g)$ onto a Riemannian manifold $(N,g_N)$. If $\mu=\{0\}$ or  $\mu=span\{\xi\}$, i.e., $ker\pi_{*}^{\bot}=\varphi (ker\pi_{*})$ or $ker\pi_{*}^{\bot}=\varphi (ker\pi_{*})\oplus<\xi>$, respectively, then we call $\pi$ a \emph{Lagrangian submersion.} \\

In that case, for any horizontal vector field $X,$ we have
\begin{equation} \label{f1}
\begin{array}{c}
\mathcal{B}X=\varphi X \quad and \quad \mathcal{C}X=0
\end{array}.
\end{equation}
For the general properties of such submersions, see \cite{Ta1,Ta2}.
\end{definition}
\subsection{Clairaut submersions}
Let $S$ be a revolution surface in $\mathbb{R}^3$ with rotation axis $L$.  For any $p\in S$, we denote by $r(p)$ the distance from $p$ to $L$. Given a geodesic $\alpha:I\subset \mathbb{R}\rightarrow S$ on $S$, let $\theta(t)$ be the angle between $\alpha(t)$ and the meridian curve through $\alpha(t)$, $t\in I$. A well-known Clairaut's theorem says that for any geodesic $\alpha$ on $S$ the product $r\,sin\theta$ is constant along $\alpha$, i.e., it is independent of $t$. In the theory of Riemannian submersions,  Bishop \cite{Bi} introduces the notion of Clairaut submersion in the following way.\\
\begin{definition}\label{defn1}(\cite{Bi})
A Riemannian submersion $\pi: (M,g)\rightarrow (N,g_{N})$ is called a \emph{Clairaut submersion} if there exists a positive function $r$ on $M$ such that, for any geodesic $\alpha$ on $M$, the function $(r\circ \alpha)sin\theta$ is constant, where,  for any $t$, $\theta(t)$ is the angle between $\dot{\alpha}(t)$ and the horizontal space at $\alpha(t)$.
\end{definition}
He also gave the following necessary and sufficient condition for a Riemannian submersion to be a Clairaut submersion as follows.
\begin{theorem}\label{dortdort}(\cite{Bi})
Let  $\pi: (M,g)\rightarrow (N,g_{\text{\tiny$N$}})$ be Riemannian submersion with connected fibers. Then $\pi$ is a Clairaut submersion with $r=e^{f}$ if and only if each fibre is totally umbilical and has the mean curvature vector field $H=-\nabla f$, where $\nabla f$ is the gradient of the function $f$ with respect to $g$.
\end{theorem}
\section{Anti-invariant submersions admitting horizontal\\
reeb vector field from sasakian manifolds}
In this section, we study anti-invariant submersions from Sasakian manifolds admitting horizontal Reeb vector field. After giving a new necessary and sufficient condition for such submersions to be Clairaut, we prove some characteristic results for this kind of submersions. We also present an illustrative example for such submersions at the end of this section.\\

As seen from Definition \ref{defn1}, the origin of the notion of a Clairaut submersion comes from geodesic on its total space. Therefore, we will investigate a necessary and sufficient condition for a curve on the total space to be geodesic.
\begin{lemma}
Let $\pi$ be an anti-invariant Riemannian submersion from a Sasakian manifold $(M,\varphi,\xi,\eta,g)$ onto a Riemannian manifold $(N, g_{\text{\tiny$N$}})$ admitting horizontal Reeb vector field. If $\alpha:I\subset \mathbb{R}\rightarrow M$ is a regular curve and $V(t)$ and $X(t)$ are the vertical and horizontal components of the tangent vector field $\dot{\alpha}(t)=E$ of $\alpha(t)$, respectively, then $\alpha$ is a geodesic if and only if along $\alpha$ the following equations
\begin{eqnarray}\label{sub1}
\mathcal{V}\nabla_{\dot{\alpha}}\mathcal{B}X+\mathcal{A}_{X}\varphi V +\mathcal{T}_{V}\varphi V+(\mathcal{T}_{V}+\mathcal{A}_{X})\mathcal{C}X +\eta(X)V=0 \,,
\end{eqnarray}
\begin{eqnarray}\label{sub2}
\mathcal{H}\nabla_{\dot{\alpha}} \mathcal{C}X + \mathcal{H}\nabla_{\dot{\alpha}} \varphi V +(\mathcal{T}_{V}+\mathcal{A}_{X})\mathcal{B} X + \eta(X)X-\nu\xi =0\,
\end{eqnarray}
hold, where $\sqrt{\nu}$ is constant speed of $\alpha.$
\end{lemma}
\begin{proof}
From \eqref{e12},
we have
\begin{eqnarray}
\nabla_{\dot{\alpha}}\varphi \dot{\alpha}=\varphi\nabla_{\dot{\alpha}}\dot{\alpha}+g(\dot{\alpha},\dot{\alpha})\xi-\eta(\dot{\alpha})\dot{\alpha}.
\end{eqnarray}
Since $\dot{\alpha}=V+X $ and $g(\dot{\alpha},\dot{\alpha})=\nu$, we can write
\[\nabla_{V+X} \varphi(V+X)=\varphi \nabla_{\dot{\alpha}}\dot{\alpha} +\nu\xi-\eta(V)\dot{\alpha}-\eta(X)\dot{\alpha}.\] By direct computations, we obtain
\[\nabla_{V} \varphi V+\nabla_{V} \varphi X+\nabla_{X} \varphi V+\nabla_{X} \varphi X= \varphi \nabla_{\dot{\alpha}} \dot{\alpha}+\nu\xi-\eta(X) V-\eta(X)X,\] since $\eta(V)=0$.\\

Using \eqref{e5}$\sim$\eqref{e8}, we get \\

$\mathcal{H}(\nabla_{\dot{\alpha}}\varphi V +\nabla_{\dot{\alpha}}\mathcal{C}X)+(\mathcal{T}_{V}+\mathcal{A}_{X})(\mathcal{B}X+\mathcal{C}X)+\mathcal{V}\nabla_{\dot{\alpha}}\mathcal{B}X+\mathcal{A}_{X} \varphi V+\mathcal{T}_{V} \varphi V$\\

$=\varphi \nabla_{\dot{\alpha}} \dot{\alpha} +\nu\xi-\eta(X)V-\eta(X)X$.\\

Taking the vertical and horizontal parts of above equation, we get
\begin{eqnarray}\label{sub3}
\mathcal{V}\nabla_{\dot{\alpha}} \mathcal{B}X + \mathcal{A}_{X} \varphi V+\mathcal{T}_{V} \varphi V+ (\mathcal{T}_{V}+\mathcal{A}_{X})  \mathcal{C}X= \mathcal{V}\varphi \nabla_{\dot{\alpha}}\dot{\alpha}-\eta(X)V
\end{eqnarray}
and
\begin{eqnarray}\label{sub4}
\mathcal{H}\nabla_{\dot{\alpha}} \mathcal{C}X +\mathcal{H}\nabla_{\dot{\alpha}} \varphi V +(\mathcal{T}_{V}+\mathcal{A}_{X})  \mathcal{B}X=\mathcal{H}\varphi \nabla_{\dot{\alpha}}\dot{\alpha}+\nu\xi-\eta(X)X
\end{eqnarray}
From \eqref{sub3} and \eqref{sub4}, it is easy to see that $\alpha$ is a geodesic if and only if  \eqref{sub1} and \eqref{sub2} hold.
\end{proof}
\begin{theorem}\label{theo3}
Let $\pi$ be an anti-invariant Riemannian submersion from a Sasakian manifold $(M,\varphi,\xi,\eta,g)$ onto a Riemannian manifold $(N, g_{\text{\tiny$N$}})$ admitting horizontal Reeb vector field. Then $\pi$ is a Clairaut submersion with $r=e^{f}$ if and only if along $\alpha$
\begin{eqnarray}\label{sub10}\qquad
g(\nabla f,X) {\parallel V\parallel}^2= g(\eta(X)X+\mathcal{H}\nabla_{\dot{\alpha}} \mathcal{C}X +(\mathcal{T}_{V}+\mathcal{A}_{X})  \mathcal{B}X, \varphi V)
\end{eqnarray}
holds, where $V(t)$ and $X(t)$ are the vertical and horizontal components of the tangent vector field $\dot{\alpha}(t)$ of the geodesic $\alpha(t)$ on $M$, respectively.
\end{theorem}
\begin{proof}
Let $\alpha(t)$ be a geodesic with  speed $\sqrt{\nu}$ on $M$, then we have \[\nu={\parallel\dot{\alpha}(t)\parallel}^{2},\]
From this equality, we deduce that
\begin{eqnarray}\label{sub5}\qquad
g(V(t), V(t))=\nu sin^2 \theta (t)\qquad  \mbox{and} \qquad g(X(t), X(t))=\nu cos^2 \theta (t),
\end{eqnarray}
where $\theta(t)$ is the angle between $\dot{\alpha}(t)$ and the horizontal space at $\alpha(t)$. Differentiating the first expression in \eqref{sub5}, we obtain
\[\frac{d}{dt}g(V(t),V(t))=2g(\nabla_{\dot{\alpha(t)}}V(t),V(t))=2\nu cos \theta(t) sin\theta(t)\frac{d\theta}{dt}(t).\]
Hence using the Sasakian structure, we get
\begin{eqnarray}\label{sub6}\qquad
 g(\varphi\nabla_{\dot{\alpha(t)}}V(t),\varphi V(t) )=\nu cos \theta (t) sin\theta(t) \frac{d\theta}{dt}(t)\,,
 \end{eqnarray}
 \\
  At this point, we know \[\varphi \nabla_{\dot{\alpha}} V=\nabla_{\dot{\alpha}}\varphi V-g(\dot{\alpha}, V)\xi\] from \eqref{e12}. Hence, \[g(\varphi\nabla_{\dot{\alpha}}V,\varphi V )=g(\nabla_{\dot{\alpha}}\varphi V, \varphi V)=g(\mathcal{H}\nabla_{\dot{\alpha}}\varphi V, \varphi V),\] since $g(\xi,\varphi V)=0$ and $\varphi V$ is horizontal.\\

Thus, from \eqref{sub6}, we obtain
\begin{eqnarray}
g(\mathcal{H}\nabla_{\dot{\alpha}}\varphi V, \varphi V)=\nu cos \theta   sin\theta \frac{d\theta}{dt}
\end{eqnarray}

By \eqref{sub2}, we find along $\alpha$,
\begin{eqnarray}\label{sub7}\qquad
-g(\mathcal{H}\nabla_{\dot{\alpha}}\mathcal{C} X+(\mathcal{T}_{V}+\mathcal{A}_{X})  \mathcal{B}X+\eta(X)X, \varphi V)=\nu cos \theta  sin\theta \frac{d\theta}{dt},
\end{eqnarray}
since $g(\xi, \varphi V)=0$.\\

On the other hand, $\pi$ is a Clairaut submersion with $r=e^{f}$ if and only if \[\frac{d}{dt}(e^{f} sin \theta)=0\Leftrightarrow e^{f} (\frac{df}{dt} sin \theta+cos \theta \frac{d\theta}{dt})=0\]

Multiplying last equation with non-zero factor $\nu sin\theta$, we get

\begin{eqnarray}\label{sub8}\qquad
\frac{df}{dt} \nu sin^{2} \theta+\nu cos\theta sin\theta \frac{d\theta}{dt}=0
\end{eqnarray}
From  \eqref{sub7} and \eqref{sub8}, we obtain
\begin{eqnarray}\label{sub9}\qquad
\frac{df}{dt}(\alpha(t)) \parallel V\parallel^{2} = g(\mathcal{H}\nabla_{\dot{\alpha}}\mathcal{C} X+(\mathcal{T}_{V}+\mathcal{A}_{X})  \mathcal{B}X-\eta(X)X, \varphi V)
\end{eqnarray}
Since $\displaystyle{ \frac{df}{dt}(\alpha(t))=\dot{\alpha}[f]=g(\nabla f,\dot{\alpha})=g(\nabla f,X)}$, the assertion \eqref{sub10} follows from \eqref{sub9}.
\end{proof}
From \eqref{sub10}, we get the following result.
\begin{corollary}\label{cor1}
Let $\pi$ be an Clairaut anti-invariant  Riemannian submersion from a Sasakian manifold $(M,\varphi,\xi,\eta,g)$ onto a Riemannian manifold $(N, g_{\text{\tiny$N$}})$ admitting horizontal Reeb vector field. Then we have
\begin{equation}\label{cor2}
g(\nabla f, \xi)=0\,.
\end{equation}
\end{corollary}
Next, we give a characterization for Clairaut anti-invariant Riemannian submersion admitting horizontal Reeb vector field.
\begin{theorem}\label{theo1}
Let $\pi$ be a  Clairaut anti-invariant Riemannian submersion  admitting horizontal Reeb vector field from a Sasakian manifold $(M,\varphi,\xi,\eta,g)$ onto a Riemannian manifold $(N, g_{\text{\tiny$N$}})$ with $r=e^{f}$. Then at least one of the following statements are true:\\

\textbf{(a)} $f$ is constant on $\varphi ker \pi_{*}$\,,\\

\textbf{(b)} the fibers of $\pi$ are one dimensional\,,\\

\textbf{(c)} $\mathcal{A}_{JW} JX=X(f)W$\\

for $X\in \mu$ and $W\in ker \pi_{*}$ such that $JW$ is basic.
\end{theorem}
\begin{proof}
Let $\pi$ be a  Clairaut anti-invariant Riemannian submersion  admitting horizontal Reeb vector field from a Sasakian manifold $(M,\varphi,\xi,\eta,g)$ onto a Riemannian manifold $(N, g_{\text{\tiny$N$}}) $ with $r=e^{f}$. From Bishop's theorem, we have

\begin{eqnarray}\label{sub11}\qquad
  \mathcal{T}_{U} V=-g(U,V) \nabla f
  \end{eqnarray}
  where $U,V\in ker \pi_{*}$, If we multiply this equation by $\varphi W$ for $W\in ker \pi_{*}$ and using \eqref{e5}, we obtain \[g(\nabla_{U}V,\varphi W)=-g(U,V)g(\nabla f,\varphi W).\]
  Hence, we get \[g(\nabla_{U}\varphi W, V)=g(U,V)g(\nabla f,\varphi W),\] since  $g(V,\varphi W)=0$.\\

By \eqref{e12}, we arrive at \[g(\varphi \nabla_{U} W, V)=g(U,V)g(\nabla f,\varphi W).\]

Using the Sasakian structure, we find \[-g(\nabla_{U} W, \varphi V)=g(U,V)g(\nabla f,\varphi W)\]

Again, using \eqref{e5}, we get \[-g(\mathcal{T}_{U} W, \varphi V)=g(U,W)g(\nabla f, \varphi V)\]

Hence, by \eqref{sub11},
   \begin{eqnarray}\label{sub12}\qquad
g(U,W)g(\nabla f, \varphi V)=g(U,V)g(\nabla f,\varphi W)
  \end{eqnarray}
If take $U=W$ and interchange $U$ with by $V$ in  \eqref{sub12}, we derive
\begin{eqnarray}\label{sub13}\qquad
  {\parallel V\parallel}^{2} g(\nabla f, \varphi U)=g(U,V)g(\nabla f, \varphi V)
  \end{eqnarray}

  Using \eqref{sub12} with $W=U$ and \eqref{sub13}, we have

  \begin{eqnarray}\label{sub14}\qquad
 g(\nabla f, \varphi U)=\frac{g^{2}(U,V)}{{\parallel U\parallel}^{2}{\parallel V\parallel}^{2}} g(\nabla f, \varphi U)
  \end{eqnarray}

  On the other hand, using \eqref{e12}, we have \[g(\nabla_{V} \varphi W,\varphi X)=g(\varphi\nabla_{V}  W ,\varphi X).\]
  for $X\in \mu$ and $X\neq \xi$. Hence, using the Sasakian structure, we obtain \[g(\nabla_{V} \varphi W,\varphi X)=g(\nabla_{V}  W , X).\]
  Using \eqref{e5} and \eqref{sub11}, we get

  \begin{eqnarray}\label{sub15}\qquad
 g(\nabla_{V} \varphi W,\varphi X)=-g(V,W)g(\nabla f, X)
  \end{eqnarray}

  Since $\varphi W$ is basic and using the fact that $\mathcal{H}\nabla_{V} \varphi W=\mathcal{A}_{\varphi W}V$, we get
  \begin{eqnarray}\label{sub16}\qquad
 g(\nabla_{V} \varphi W,\varphi X)=g(\mathcal{A}_{\varphi W}V,\varphi X)
  \end{eqnarray}

 Using \eqref{sub15}$\sim$\eqref{sub16} and the skew-symmetricness of $\mathcal{A}$, we find
   \begin{eqnarray}\label{sub17}\qquad
 g(\mathcal{A}_{\varphi W}\varphi X,V)=g(\nabla f, X)g(W,V).
  \end{eqnarray}

  Since $\mathcal{A}_{\varphi W}\varphi X,V$ and $W$ are vertical and $\nabla f$ is horizontal, we  deduce that

  \begin{eqnarray}\label{sub18}\qquad
\mathcal{A}_{\varphi W}\varphi X=X(f)W
  \end{eqnarray}
  from \eqref{sub17}.\\

Now, if $\nabla f \in \varphi ker \pi_{*}$, then \eqref{sub14} and the equality case of Schwarz inequality imply that either $f$ is constant on $\varphi ker \pi_{*}$ or the fibers of $\pi$ are one dimensional. Thus \textbf{(a)} and \textbf{(b)} follows. If $\nabla f\in \mu\setminus \{\xi\}$, the last assertion follows immediately from \eqref{sub18}.
\end{proof}
\begin{corollary}
Let $\pi$ be a  Clairaut anti-invariant Riemannian submersion admitting horizontal Reeb vector field from a Sasakian manifold $(M,\varphi,\xi,\eta,g)$ onto a Riemannian manifold $(N, g_{\text{\tiny$N$}})$ with $r=e^{f}$ and $dim(ker \pi_{*})>1.$ Then the fibers of $\pi$ are totally geodesic if and only if $\mathcal{A}_{ JV} JX=0$ for $V\in \Gamma(ker \pi_{*})$ such that $JV$ is basic and $X\in \mu$.
\end{corollary}

Moreover, if the submersion $\pi$ in Theorem \ref{theo1} is Lagrangian, then $\mathcal{A}_{JV} JX$ is always zero, since $\mu =\{0\}$ or $\mu=span \{\xi\}$. Thus, we have the following result from Theorem \ref{theo1}.

\begin{corollary}
Let $\pi$ be a  Clairaut Lagrangian  submersion admitting horizontal Reeb vector field from a Sasakian manifold $(M,\varphi,\xi,\eta,g)$ onto a Riemannian manifold $(N, g_{\text{\tiny$N$}})$ with $r=e^{f}$. Then either the fibers of $\pi$ are one dimensional or they are totally geodesic.
\end{corollary}
We ends this section by giving a (non-trivial) example of a Clairaut anti-invariant submersion from Sasakian manifold admitting horizontal Reeb vector field.
\begin{example}\label{ex5}
Let $\overline{\mathbb{R}}^{3}$ be 3-dimensional Euclidean space given by \[\overline{\mathbb{R}}^{3}=\{(x,y,z)\in \mathbb{R}^3\,\mid\, (x,y)\neq(0,0)\; \textrm{and}\, z\neq 0\}.\]
We consider the map $\pi:(\overline{\mathbb{R}}^{3},\varphi_{0},\xi,\eta,g)\rightarrow(\mathbb{R}^{2}, g_{2})$ defined by \[\pi(x,y,z)=(\sqrt{x^2 +y^2},z)\]
where $(\varphi_{0},\xi,\eta,g)$ is the usual Sasakian structure \cite{Ca} on $\overline{\mathbb{R}}^{3}$ and $g_{2}$
is the Euclidean metric on $\mathbb{R}^{2}.$ Then the Jacobian matrix of $\pi$ is
\[\left(
 \begin{array}{ccc}
x/\tau&y/\tau&0\\
 0&0&1
 \end{array}\right).\]
 Here, $\tau=\sqrt{x^2 + y^2}.$ Since the rank of this matrix is equal to 2, the map $\pi$ is a submersion. Following some computations, we have \[ker\pi_{*}=span\{U=\frac{y}{\tau}E_1 - \frac{x}{\tau}E_2\}\]
 and
 \[ker\pi_{*}^\bot =span\{Z=\frac{x}{\tau}E_1 +\frac{y}{\tau}E_2\,, E_3=\xi\},\]
 where $\{E_1, E_2, E_3\}$ is a $\varphi_0 $-basis such that ${E_1  =2(\frac{\partial}{\partial x}+y\frac{\partial}{\partial z})}$
 and ${E_2 =2\frac{\partial}{\partial y}}.$\\

For this map $\pi$, it is not difficult to satisfy the condition \textbf{S2)}. So, $\pi$ is a Riemannian submersion.  Also, we have $\varphi_0 (U)=-X$. Hence, we see that $\pi$ is an anti-invariant Riemannian submersion admitting horizontal Reeb vector field. In particular, $\pi$ is Lagrangian.
Moreover, since the fibers of $\pi$ are one dimensional, they are clearly totally umbilical. Here, we shall show that the fibers are not totally geodesic and find that a function on $\overline{\mathbb{R}}^3$ satisfying $\mathcal{T}_U U=-\nabla f.$ Indeed, by direct computations, we have
\begin{equation}\label{ex4}
\nabla_U U=U[\frac{1}{\tau}]\tau  U-\frac{2}{\tau^2}(xE_1 +yE_2)+\bigg (\frac{y}{\tau}U[E_1]-\frac{x}{\tau}U[E_2]\bigg )
\end{equation}
Here, one can see that \[U[E_1]=\frac{y}{\tau}\nabla_{E_1} E_1 -\frac{x}{\tau}\nabla_{E_2} E_1\] and
\[U[E_2]=\frac{y}{\tau}\nabla_{E_1} E_2-\frac{x}{\tau}\nabla_{E_2} E_2.\]
Using  the Sasakian structure, we see that
\[\nabla_{E_1} E_1=\nabla_{E_2} E_2=0\] and \[\nabla_{E_1} E_2=-\nabla_{E_2} E_1=-2\frac{\partial}{\partial z}.\]
Taking into account these equalities in \eqref{ex4}, we obtain
\[\nabla_U U=U[\frac{1}{\tau}]\tau V-\frac{2}{\tau^2}(xE_1 +yE_2).\]
Using \eqref{e5}, we get \[\mathcal{T}_U U=-\frac{2}{\tau^2}(xE_1 +yE_2).\] After some calculation, we arrive
\[T_U U=-\bigg \{\frac{2x}{x^2 +y^2}\frac{\partial}{\partial x}+\frac{2y}{x^2 +y^2}\frac{\partial}{\partial y}+ \frac{2xy}{x^2 +y^2}\frac{\partial}{\partial z}\bigg \}.\]
For any function $f$ on $(\overline{\mathbb{R}}^{3},\varphi_0 ,\xi , \eta ,g)$, the gradient of $f$ with respect to the metric $g$ is:\\
$\displaystyle{\nabla f=\sum_{i,j} g^{ij}\frac{\partial f}{\partial x_i}\frac{\partial}{\partial x_j}=4\bigg \{\bigg(\frac{\partial f}{\partial x}+y\frac{\partial f}{\partial z}\bigg)\frac{\partial}{\partial x}+\frac{\partial f}{\partial y}\frac{\partial }{\partial y}+\bigg(y\frac{\partial f}{\partial x}+(1+y^2)\frac{\partial f}{\partial z}\bigg)\frac{\partial }{\partial z}\bigg \}.}$\\
Then, for the function $f=\frac{1}{4}\ln(x^2 +y^2),$ it is easy to verify that \[\mathcal{T}_U U=-\nabla f.\]

Hence, it follows that
\begin{equation}\nonumber
\mathcal{T}_V V=-\|V\|^{2}\nabla f
\end{equation}
for any vertical vector field $V.$ Under the given conditions, the tensor $\mathcal{T}$ is never zero. So, the fibers of $\pi$ are not totally geodesic, but they are totally umbilical with mean curvature field $H=-\nabla f.$ Thus, by Theorem \ref{dortdort}, we see that this anti-invariant Riemannian submersion  is Clairaut with $r=e^{f}$, where $f=\frac{1}{4}\ln(x^2 +y^2).$
\end{example}
Henceforth, we have alternative theorem, namely Theorem \ref{theo3}, to check that whether the submersion is Clairaut or not.\\

In fact, for any horizontal vector field $X$ proportional to $\xi$, we easily verify that the condition \eqref{sub10} of Theorem \ref{theo3}.\\

Now, let $X$ be any horizontal vector field orthogonal to $\xi$ and $V$ be any vertical vector field,
then using \eqref{new1}, \eqref{f1} and the equation (34) of Corollary 6.1 of \cite{Ta2}, we have
\begin{align*}
g(\mathcal{T}_V\mathcal{B}X,\varphi V)=&g(\mathcal{T}_V\varphi X, \varphi V)\\
=&g(\varphi\mathcal{T}_VX,\varphi V)\\
=&g(\mathcal{T}_VX,V)\\
=&-g(\mathcal{T}_VV,X).
\end{align*}
Hence, we obtain
\begin{equation} \label{nex1}
g(\mathcal{T}_V\varphi X,\varphi V)=g(\nabla f,X)\|V\|^{2}\,,
\end{equation}
since $\mathcal{T}_V V=-\|V\|^{2}\nabla f.$ Likewise, using \eqref{new1}, \eqref{f1} and the equation (35) of
Corollary 6.1 of \cite{Ta2}, we have
\begin{align*}
g(\mathcal{A}_X\mathcal{B} X,\varphi V)=&g(\mathcal{A}_X\varphi X,\varphi V)\\
=&-g(\varphi\mathcal{A}_X\varphi V,\varphi X)\\
=&-g(\varphi\mathcal{A}_XV,X)\\
=&-g(\mathcal{A}_XX,V).
\end{align*}
Hence, we obtain
\begin{equation} \label{nex2}
g(\mathcal{A}_X\varphi X,\varphi V)=0\,,
\end{equation}
since $\mathcal{A}_XX=0.$ In addition to, we have
\begin{equation} \label{nex3}
\begin{array}{c}
\eta(X)=0 \quad and \quad \mathcal{H}\nabla_{\dot{\alpha}} \mathcal{C}X=0
\end{array},
\end{equation}
since $\pi$ is Lagrangian and $X$ is orthogonal to $\xi.$ Using \eqref{nex1}, \eqref{nex2} and \eqref{nex3}, we easily verify the equation
\eqref{sub10}. Thus, by Theorem \ref{theo3}, the considered submersion $\pi$ is Clairaut.
\begin{remark}
We notice that the submersion given in Example \ref{ex5} satisfies one of the conditions of
Theorem \ref{theo1} and the condition \eqref{cor2} in Corollary \ref{cor1}.
\end{remark}
\section{Anti-invariant submersions admitting vertical\\
reeb vector field from sasakian manifolds}
In this section, we check that the existence of Clairaut anti-invariant submersions from Sasakian manifolds when the Reeb vector field is vertical.
First of all, we give a non-trivial example of an anti-invariant submersion from Sasakian manifold admitting vertical Reeb vector field.
\begin{example}
Let $\mathbb{R}^5$ be a Sasakian manifold with usual Sasakian structure \cite{Ca}. Consider the map $\pi: \mathbb{R}^5\rightarrow(\mathbb{R}^{2}, g_{2})$ given by
\[\pi(x_1,x_2,y_1,y_2,z)=(\frac{x_1 +y_1}{\sqrt{2}},\frac{x_2 +y_2}{\sqrt{2}})\,,\]
where $g_{2}$ is the Euclidean metric on $\mathbb{R}^{2}.$ After some calculation, we see that
\[ker\pi_{*}=span \{V=\frac{1}{\sqrt{2}}(E_3 -E_1),\;W=\frac{1}{\sqrt{2}}(E_4 -E_2),\;\xi\}\]
and
\[ker\pi_{*}^\bot=span\{X=\frac{1}{\sqrt{2}}(E_1 +E_3),\; Y=\frac{1}{\sqrt{2}}(E_2 +E_4)\}\]
It is not difficult to show that $\pi$ is a Riemannian submersion. Also, we have $\varphi_0 (V)=-X$ and $\varphi_0 (W)=-Y.$
Hence, $\pi$ is an anti-invariant submersion admitting vertical Reeb vector field. In particular, $\pi$ is Lagrangian.
\end{example}
We now assume that there exists an anti-invariant submersion $\pi$ admitting vertical Reeb vector field from Sasakian manifold satisfying Clairaut condition. Then because of Theorem \ref{dortdort}, the fibers of $\pi$ must be totally umbilical. But, the following result forces the fibers to be totally geodesic, since the fibers are submanifolds.
\begin{theorem} (\cite{I})
$\;$Let $\widetilde{N} $ be a Sasakian manifold. If $N$ is any totally umbilical submanifold of $\widetilde{N}$ tangent to the Reeb vector field $\xi$, then it is totally geodesic.
\end{theorem}
On the other hand, for any vertical vector field V, we have \begin{equation}\label{besbir}\mathcal{T}_V \xi=-\varphi V\end{equation} from the proof of Theorem 2 of \cite{E}. The equation \eqref{besbir} says us the fibers of $\pi$ cannot be totally geodesic. This is a contradiction. Thus, we have the following classification theorem.
\begin{theorem}
There is no Clairaut anti-invariant submersions admitting vertical Reeb vector field from Sasakian manifolds onto Riemannian manifolds.
\end{theorem}

\section{Anti-invariant submersions from kenmotsu manifolds}
In this section, we shall give new Clairaut conditions for anti-invariant Riemannian submersions from Kenmotsu manifolds onto Riemannian manifolds. In which case,
the Reeb vector field $\xi$ is necessarily horizontal, because Beri et al. \cite{Be} showed the non-existence of anti-invariant Riemannian submersions from Kenmotsu manifolds such that the Reeb vector field is vertical.
\begin{lemma}
Let $\pi$ be an anti-invariant  Riemannian submersion  from a Kenmotsu manifold $(M,\varphi,\xi,\eta,g)$ onto a Riemannian manifold $(N, g_{\text{\tiny$N$}})$. If $\alpha:I\subset \mathbb{R}\rightarrow M$ is a regular curve and $V(t)$ and $X(t)$ are the vertical and horizontal components of the tangent vector field $\dot{\alpha}(t)=E$ of $\alpha(t)$, respectively, then $\alpha$ is a geodesic if and only if the following two equations

\begin{eqnarray}\label{sub28}
\mathcal{V}\nabla_{\dot{\alpha}}\mathcal{B}X +\mathcal{A}_{X}\varphi V+(\mathcal{T}_{V} +\mathcal{A}_{X})\mathcal{C}X+\eta(X)\mathcal{B}X=0 \,,
\end{eqnarray}
\begin{eqnarray}\label{sub29}
\mathcal{H}\nabla_{\dot{\alpha}}(\varphi V+\mathcal{C}X) +(\mathcal{T}_{V}+\mathcal{A}_{X})\mathcal{B}X +\eta(X)(\varphi V +CX)=0 \,
\end{eqnarray}
hold along $\alpha$.
\end{lemma}
\begin{proof}
From \eqref{f1}, we have \[\nabla_{\dot{\alpha}}\varphi \dot{\alpha}=\varphi \nabla_{\dot{\alpha}} \dot{\alpha}+g(\varphi \dot{\alpha},\dot{\alpha})\xi-\eta(\dot{\alpha})\varphi \dot{\alpha}.\]
Since $\dot{\alpha}=V+X$ and $\eta(V)=0$, we can write \[\nabla_{V}\varphi V+\nabla_{V}\varphi X+\nabla_{X}\varphi V+\nabla_{X}\varphi X=\varphi \nabla_{\dot{\alpha}}\dot{\alpha}-\eta(X)(\varphi V +\varphi X),\]
Using \eqref{e5}$\sim$\eqref{e8}, together with \eqref{f0}, we obtain
\[\mathcal{H}\nabla_{\dot{\alpha}}(\varphi V+\mathcal{C}X)+(\mathcal{T}_{V}+\mathcal{A}_{X})(\mathcal{B}X+\mathcal{C}X)+\mathcal{V}\nabla_{\dot{\alpha}}\mathcal{B}X+\mathcal{A}_{X} \varphi V\]
\[=\varphi \nabla_{\dot{\alpha}}\dot{\alpha}-\eta(X)(\mathcal{B}X+\mathcal{C}X+\varphi V)\]
Taking the vertical and horizontal parts of the last equation, we get
\begin{eqnarray}\label{sub30}\;\;
\mathcal{V}\nabla_{\dot{\alpha}}\mathcal{B}X +\mathcal{A}_{X}\varphi V+(\mathcal{T}_{V}+ +\mathcal{A}_{X})\mathcal{C}X=\mathcal{V}\varphi \nabla_{\dot{\alpha}}\dot{\alpha}-\eta(X)\mathcal{B}X \,,
\end{eqnarray}
\begin{eqnarray}\label{sub31}\;\;
\mathcal{H}\nabla_{\dot{\alpha}}(\varphi V+\mathcal{C}X) +(\mathcal{T}_{V}+\mathcal{A}_{X})\mathcal{B}X =\mathcal{H}\varphi \nabla_{\dot{\alpha}}\dot{\alpha}-\eta(X)(\mathcal{C}X+\varphi V) \,,
\end{eqnarray}
From \eqref{sub30} and \eqref{sub31}, we see that $\alpha$ is a geodesic if and only if \eqref{sub28} and \eqref{sub29} hold along $\alpha$.
\end{proof}
\begin{theorem}\label{theo2}
Let $\pi$ be an anti-invariant  Riemannian submersion  from a Kenmotsu manifold $(M,\varphi,\xi,\eta,g)$ onto a Riemannian manifold $(N, g_{\text{\tiny$N$}})$. Then, $\pi$ is a Clairaut submersion with $r=e^{f}$ if and only if
\begin{eqnarray}\label{sub32}
\{g(\nabla f, X)-\eta(X)\}{\parallel V\parallel}^{2}=g(\mathcal{H}\nabla_{\dot{\alpha}}\mathcal{C}X +(\mathcal{T}_{V}+\mathcal{A}_{X})\mathcal{B}X,\varphi V )
\end{eqnarray}
 holds along $\alpha$, where $V(t)$ and $X(t)$ are the vertical and horizontal components of the tangent vector field $\dot{\alpha}(t)$ of the geodesic $\alpha(t)$ on $M$, respectively.
\end{theorem}
\begin{proof}
Let $\alpha$  be a geodesic on $M$, then we have \[{\parallel \dot{\alpha}(t)\parallel}^{2}=\nu,\]
where $c$ is a constant. Hence, we deduce that
\begin{eqnarray}\label{sub33}
g(V(t),V(t))=\nu sin^{2}\theta(t)\qquad and \qquad g(X(t),X(t))=\nu cos^{2}\theta(t)
\end{eqnarray}
where $\theta(t)$ is the angle between $\dot{\alpha}(t)$ and the horizontal space at $\alpha(t)$. Differentiating the first expression , we obtain
\begin{eqnarray}\label{sub34}
g(\nabla_{\dot{\alpha(t)}}V(t),V(t))=\nu cos\theta(t) sin\theta(t)\frac{d\theta}{dt}(t)
\end{eqnarray}
Using  the Kenmotsu structure, we get
\[
g(\nabla_{\dot{\alpha}}V, V)=g(\varphi \nabla_{\dot{\alpha}}V,\varphi V)
,\]
since $\eta(V)=0$. Here, by \eqref{f1}, we know \[\varphi \nabla_{\dot{\alpha}}V=\nabla_{\dot{\alpha}}\varphi V-g(\varphi\dot{\alpha},V )\xi. \]
Hence, we obtain
\begin{eqnarray}\label{sub35}\qquad
g(\varphi \nabla_{\dot{\alpha}} V,\varphi V)=g(\mathcal{H}\nabla_{\dot{\alpha}} \varphi V,\varphi V),
\end{eqnarray}
since $\varphi V$ is horizontal. From \eqref{sub34} and \eqref{sub35}, we get
\begin{eqnarray}\label{sub36}\qquad
g(\mathcal{H}\nabla_{\dot{\alpha}} \varphi V,\varphi V)=\nu cos\theta sin\theta\frac{d\theta}{dt}.
\end{eqnarray}
Using \eqref{sub29}, we find
\begin{eqnarray}\label{sub37}\qquad
-g(\mathcal{H}\nabla_{\dot{\alpha}}\mathcal{C}X+ (\mathcal{T}_{V}+\mathcal{A}_{X})\mathcal{B}X+\eta(X)\varphi V, \varphi V )=\nu cos\theta sin\theta\frac{d\theta}{dt}
\end{eqnarray}
As in the proof of Theorem \ref{theo3}, $\pi$ is a Clairaut submersion with $r=e^{f}$ if and only if  \eqref{sub8} holds.
Thus, from \eqref{sub8} and \eqref{sub37}, we get
\begin{eqnarray}\label{sub38}
\frac{d(f\circ\alpha)}{dt}{\parallel V\parallel} ^{2}=g(\mathcal{H}\nabla_{\dot{\alpha}}\mathcal{C}X+ (\mathcal{T}_{V}+\mathcal{A}_{X})\mathcal{B}X,\varphi V)
\end{eqnarray}
\phantom{ghsghchciihihlhpkhklhmhhhhs}$+\,\eta(X){\parallel V\parallel} ^{2}$\\
Since  $\;\displaystyle{\frac{d(f\circ \alpha)}{dt}=g(\nabla f,X)}$, the assertion immediately follows  from  \eqref{sub38}.
\end{proof}
From \eqref{sub32}, we immediately have that:
\begin{corollary} \label{cor4} Let $\pi$ be a Clairaut anti-invariant  Riemannian submersion  from a Kenmotsu manifold $(M,\varphi,\xi,\eta,g)$ onto a Riemannian manifold $(N, g_{\text{\tiny$N$}})$. Then,  we have
\begin{equation} \label{c4}
g(\nabla f, \xi)=1.\end{equation}
\end{corollary}
\begin{example} \label{ex3}
Let $M$ be a 3-dimensional Euclidean space given by \[M=\{(x,y,z)\in\mathbb{R}^3\mid (x,y)\neq (0,0)\; \textrm{and}\;z\neq 0\}.\]
Following the Example 1 of \cite{Be}, we define the Kenmotsu structure $(\varphi,\xi,\eta,g)$ on $M$ given by

 \[\xi=\frac{\partial}{\partial z},\; \eta=dz,\;g=\left(
 \begin{array}{lll}
 e^{2z}&0&0\\
 0&e^{2z}&0\\
 0&0&1
 \end{array}\right)\;\textrm{and}\; \varphi=\left(
 \begin{array}{cll}
 0&1&0\\
 -1&0&0\\
 0&0&0
 \end{array}\right)\]
 A $\varphi$-basis for this structure can be given by ${\{E_1=e^{-z}\frac{\partial}{\partial y},E_2=e^{-z}\frac{\partial}{\partial x},E_3=\xi\}}$.\\
 Let $N$ be $\{(u,z)\in\mathbb{R}^2\mid z\neq 0\}.$ We choose the Riemannian metric $g_N$ on $N$ in the following form
  \[\left(\begin{array}{ll}
  e^{2z}&0\\
  0&1
 \end{array}\right).\]
 Now, we define the map $\pi: (M,\varphi, \xi,\eta,g)\rightarrow (N,g_N)$ by
$$\pi(x,y,z)=(\frac{x+y}{\sqrt{2}},z).$$
Then the Jacobian matrix of $\pi$ is
 \[\left(\begin{array}{ccc}
 \frac{1}{\sqrt{2}}&\frac{1}{\sqrt{2}}&0\\
 0&0&1
 \end{array}\right).\]
Since the rank of this matrix is equal to 2, the map $\pi$ is a submersion. After simple calculations, we see that
\[ker\pi_{*}=span\{U=\frac{E_1-E_2}{\sqrt{2}}\}\,\;\textrm{and}\,\; ker\pi_{*}^\bot=span\{Z=\frac{E_1+E_2}{\sqrt{2}},\;Y=\xi\}.\]
By direct calculation, we see that $\pi$ satisfies the condition \textbf{S2)} and $\varphi(U)=-X.$ Thus, $\pi$ is an anti-invariant Riemannian submersion. In particular, $\pi$ is Lagrangian. Moreover, the fibers of $\pi$ are clearly totally umbilical, since they are one dimensional. Here, we shall find that a function $f$ on $M$ satisfying $\mathcal{T}_U U=-\nabla f.$\\
Indeed, upon direct computations, we have \[\nabla_U U=\frac{1}{2}(\nabla_{E_1}E_1-\nabla_{E_1}E_2-\nabla_{E_2}E_1+\nabla_{E_2}E_2).\]
Using the given Kenmotsu structure, we find \[\nabla_{E_1}E_1=\nabla_{E_2}E_2=-\frac{\partial}{\partial z}\] and \[\nabla_{E_1}E_2=\nabla_{E_2}E_1=0.\]
Thus, we have \[\nabla_U U=-\frac{\partial}{\partial z}.\] By \eqref{e5}, we obtain \[\mathcal{T}_U U=-\frac{\partial}{\partial z}.\]
On the other hand, for any function $f$ on $M$, the gradient of $f$ with respect to the metric $g$ is given by \[\nabla f=\sum_{i,j}^{3} g^{ij}\frac{\partial f}{\partial x_i}\frac{\partial}{\partial x_j}=e^{-2z}\frac{\partial f}{\partial x}\frac{\partial }{\partial x}+e^{-2z}\frac{\partial f}{\partial y}\frac{\partial }{\partial y}+\frac{\partial f}{\partial z}\frac{\partial }{\partial z}.\]
Then, it is easy to see that  ${\nabla f=\frac{\partial }{\partial z}}$  for the function $f=z$ and $\mathcal{T}_U U=-\nabla f=-\xi.$\\
Furthermore, for any vertical vector field $V$, we conclude that
 \begin{equation}\nonumber
\mathcal{T}_V V=-\|V\|^{2}\nabla f
\end{equation}
from the last fact. Thus, by Theorem \ref{dortdort}, the submersion $\pi$ is Clairaut.\\

Now, by using our result Theorem \ref{theo2}, we show that the submersion $\pi$ is Clairaut.\\
Indeed, if $X$ is any horizontal vector field proportional to $\xi$, then it is easy to see that the condition \eqref{sub32} is fulfilled.
Next, let $X$ be any horizontal vector field orthogonal to $\xi$ and $V$ be any vertical vector field,
then using \eqref{new1}, \eqref{f1} and the equation (59) of Corollary 7.2 of \cite{Ta2}, we have
\begin{align*}
g(\mathcal{T}_V\mathcal{B} X,\varphi V)=&g(\mathcal{T}_V\varphi X,\varphi V)\\
=&g(\varphi\mathcal{T}_VX+g(\varphi V,X)\xi-\eta(X)\varphi V,\varphi V)\\
=&g(\varphi\mathcal{T}_VX,\varphi V)\\
=&g(\mathcal{T}_VX,V)\\
=&-g(\mathcal{T}_VV,X)\\
=&g(\nabla f,X)\|V\|^{2}.\\
\end{align*}
Hence, we obtain
\begin{equation} \label{last1}
g(\mathcal{T}_V\mathcal{B} X,\varphi V)=0\,,
\end{equation}
since $\nabla f=\xi.$ Additionally, by Theorem 7.3 of \cite{Ta2}, we have $\mathcal{A}\equiv0,$
since $\pi$ is Lagrangian. Then, using this fact, \eqref{last1} and \eqref{f1}, the condition \eqref{sub32} is fulfilled.
Thus, by Theorem \ref{theo2}, the given submersion is Clairaut.
\end{example}
\begin{remark}
We notice that the Clairaut Lagrangian submersion given in Example \ref{ex3} satisfies the condition \eqref{c4} of
Corollary \ref{cor4}.
\end{remark}


\begin{thebibliography}{99}
\bibitem{Al} Allison, D. \emph{Lorentzian Clairaut submersions}, Geom. Dedicata \textbf{63} (3), (1996), 309-319.
\bibitem{Ba}  Baird, P. and  Wood, J.C.  \textit{Harmonic morphism between Riemannian manifolds}, Oxford science publications, 2003.
\bibitem{Bla}  Blair, D.E.  \textit{Contact manifolds in Riemannian geometry}, Lecture Notes in math., Springer Verlag, Berlin-New york, 509 (1976).
\bibitem{Be}  Beri, A., Erken, \.I.K. and  Murathan, C.  \textit{Anti-invariant Riemannian submersions from Kenmotsu manifolds onto Riemannian manifolds},
Turk. J. Math. (accepted)(2015), DOI 10.3906/mat-1504-47.
\bibitem{Bi}  Bishop, R.L.  \textit{Clairaut submersions}, Differential Geometry(in Honor of Kentaro Yano), Kinokuniya, Tokyo,(1972), 21-31.
\bibitem{Ca}  Cabrerizo, J.L., Carriazo A., Fern\'{a}ndez, L.M. and Fern\'{a}ndez M.  \textit{Slant submanifolds in Sasakian manifolds},
Glasgow. Math. J. \textbf{42}(2000) 125-138.
\bibitem{E} Murathan, C.  and  K\"{u}peli Erken, \.I., \textit{Anti-invariant Riemannian submersions from Cosymplectic manifolds onto Riemannian manifolds},
Filomat \textbf{29}(7)(2015) 1429-1444.
\bibitem{Fa}  Falcitelli, M.,  Ianus, S.  and  Pastore, A.M.  \textit{Riemannian submersions and related topics}, World Scientific, River Edge, NJ, (2004).
\bibitem{Gr}  Gray, A. \textit{Pseudo-Riemannian almost product manifolds and submersion}, J. Math. Mech., \textbf{16} (1967) 715-737.
\bibitem{I} Ishara, I. and  Kon, M.,  \textit{Contact totally umbilical submanifolds of a Sasakian space form}, Annali di Mathematica Pura ed Applicata, \textbf{114}(1) (1977), 351-364.
\bibitem{Ke}  Kenmotsu, K.  \textit{A class of almost contact Riemannian manifolds},
T\^{o}hoku. Math. J. \textbf{31}(1979) 247-253.
\bibitem{Er} K\"{u}peli Erken, \.I.  and  Murathan,  C.,   \textit{Anti-invariant Riemannian submersions from Sasakian manifolds}, arxiv: 1302.4906v1 [math.DG],  20 Feb 2013.
\bibitem{Le1}  Lee, J.W. \textit{Anti-invariant $\xi^{\perp}$-Riemannian submersions from almost contact manifolds}, Hacettepe
Journal of Mathematics and Statistic, \textbf{42}(3)2013() 231-241.
\bibitem{Le2} Lee, J.W.  and  \d{S}ahin, B. \textit{Pointwise slant submersions}, Bull. Korean Math. Soc. \textbf{51} (2014) 1115-1126.
\bibitem{Le3}  Lee, J.,  Park, J.H.,   \d{S}ahin, B. and  Song, D.Y. \textit{Einstein conditions for the base of anti-invariant Riemannian submersions and Clairaut submersions}, Taiwanese J. Math., 19(4)(2015), 1145-1160.
\bibitem{O}  O'Neill, B. \textit{The fundamental equations of a submersion},
Mich. Math. J. \textbf{13}(1966) 458-469.
\bibitem{Pa}  Park, K.S. and Prasad,  R. \textit{Semi-slant submersions}, Bull. Korean Math. Soc. \textbf{50}(3) (2013) 951-962.
\bibitem{Sa1}  \d{S}ahin, B. \textit{Riemannian submersions, Riemannian maps in Hermitian Geomerty, and their Applications}, Elsiever, 2017.
\bibitem{Sa2}  \d{S}ahin, B. \textit{Anti-invariant Riemannian submersions from almost Hermitian manifolds},
Cent. Eur. J. Math. \textbf{8}(3)(2010) 437-447.
\bibitem{Sa3}  \d{S}ahin, B. \textit{Semi-invariant submersions from almost Hermitian manifolds},
Canadian. Math. Bull. \textbf{56}(1)(2013) 173-182.
\bibitem{Sa4}  \d{S}ahin, B. \textit{Slant submersions from almost Hermitian manifolds},
Bull. Math. Soc. Sci. Math. Roumanie \textbf{54}(102)(2011) No. 1, 93-105.
\bibitem{Sa5}  \d{S}ahin, B., Ta\c{s}tan, H.M.  \textit{Clairaut submersions from almost Hermitian manifolds}, (preprint).
\bibitem{Ta1}  Ta\c{s}tan, H.M. \textit{On Lagrangian submersions} Hacettepe J. Math. Stat. \textbf{43}(6)(2014), 993-1000.
\bibitem{Ta2}  Ta\c{s}tan, H.M. \textit{Lagrangian submersions from normal almost contact manifolds}, Filomat, (accepted).
\bibitem{Yan}  Yano, K. and   Kon, M. \textit{Structures on manifolds}, World Scientific, Singapore, 1984.
\end{thebibliography}
\end{document}